\documentclass[smallcondensed]{cssl} 
\smartqed                                      
\usepackage{mathptmx}      
\usepackage{amssymb,latexsym}
\usepackage[colorlinks]{hyperref}

\font\teneuf=eufm10
\font\seveneuf=eufm7
\font\fiveeuf=eufm5
\font\tenmsa=msam10
\font\sevenmsa=msam7
\font\fivemsa=msam5
\font\tenmsb=msbm10
\font\sevenmsb=msbm7
\font\fivemsb=msbm5
\newfam\euffam
\textfont\euffam=\teneuf
\scriptfont\euffam=\seveneuf
\scriptscriptfont\euffam=\fiveeuf
\def\hexnumber@#1{\ifcase#1 0\or1\or2\or3\or4\or5\or6\or7\or8\or9\or
	A\or B\or C\or D\or E\or F\fi }
\newfam\msafam
\newfam\msbfam
\textfont\msafam=\tenmsa  \scriptfont\msafam=\sevenmsa
  \scriptscriptfont\msafam=\fivemsa
\textfont\msbfam=\tenmsb  \scriptfont\msbfam=\sevenmsb
  \scriptscriptfont\msbfam=\fivemsb
\edef\msb@{\hexnumber@\msbfam}
\mathchardef\varGamma="0100
\mathchardef\varTheta="0102
\mathchardef\varLambda="0103
\mathchardef\varXi="0104
\mathchardef\varPi="0105
\mathchardef\varSigma="0106
\mathchardef\varUpsilon="0107
\mathchardef\varPhi="0108
\mathchardef\varOmega="010A

\newcommand{\cf}{C}
\newcommand{\bbF}{\mathrm{I\!F}}
\newcommand{\bbH}{\mathrm{I\!H}}

\newcommand{\bbR}{\mathrm{I\!R}}

\newcommand{\bbC}{{\mathchoice {\setbox0=\hbox{$\displaystyle\mathrm{C}$}
\hbox{\hbox to0pt{\kern0.4\wd0\vrule height0.9\ht0\hss}\box0}} 
{\setbox0=\hbox{$\textstyle\mathrm{C}$}\hbox{\hbox 
to0pt{\kern0.4\wd0\vrule height0.9\ht0\hss}\box0}} 
{\setbox0=\hbox{$\scriptstyle\mathrm{C}$}\hbox{\hbox 
to0pt{\kern0.4\wd0\vrule height0.9\ht0\hss}\box0}} 
{\setbox0=\hbox{$\scriptscriptstyle\mathrm{C}$}\hbox{\hbox 
to0pt{\kern0.4\wd0\vrule height0.9\ht0\hss}\box0}}}}
\newcommand{\dimr}{\dim_{\hskip-.4pt\bbR\hskip-1.2pt}^{\phantom j}}

\newcommand{\hyp}{\hskip.5pt\vbox
{\hbox{\vrule width2.5ptheight0.5ptdepth0pt}\vskip2pt}\hskip.5pt}

\newcommand{\n}{n}
\newcommand{\p}{p}
\newcommand{\q}{q}
\newcommand{\x}{x}
\newcommand{\y}{y}
\newcommand{\z}{z}

\newcommand{\gj}{{G}}

\newcommand{\rj}{{R}}

\newcommand{\tj}{{T}}
\newcommand{\by}{\beta}
\newcommand{\byh}{\by\hskip-.4pt^\fh}
\newcommand{\gy}{\gamma}
\newcommand{\zy}{\zeta}

\newcommand{\ky}{\kappa}
\newcommand{\ly}{\lambda}
\newcommand{\my}{\mu}
\newcommand{\sy}{\sigma}
\newcommand{\ty}{\tau}
\newcommand{\xy}{\xi}

\newcommand{\cu}{\Omega}

\newcommand{\ee}{\mathcal{E}}
\newcommand{\ef}{\mathcal{F}}

\newcommand{\ep}{\mathcal{P}}

\newcommand{\w}{^{\phantom i}}

\newcommand{\hs}{\hskip.7pt}
\newcommand{\hh}{\hskip.4pt}
\newcommand{\hn}{\hskip-.4pt}
\newcommand{\nh}{\hskip-.7pt}
\newcommand{\nnh}{\hskip-1.5pt}

\newcommand{\vd}{\delta}

\newcommand{\m}{\Lambda}

\newcommand{\fg}{\mathfrak{g}}
\newcommand{\fh}{\mathfrak{h}}
\newcommand{\fy}{\mathfrak{g}}
\newcommand{\fyi}{\fy_i}

\newcommand{\fsl}{\mathfrak{sl}}
\newcommand{\fso}{\mathfrak{so}}
\newcommand{\fsp}{\mathfrak{sp}}
\newcommand{\fsu}{\mathfrak{su}}

\journalname{\ }

\begin{document}

 \newtheorem{thm}{Theorem}[section]
 \newtheorem{lem}[thm]{Lemma}
 \newtheorem{rem}[thm]{Remark}{\itshape}{\rmfamily}

\renewcommand{\theequation}{\thesection.\arabic{equation}}

\title{Curvature spectra of simple Lie groups}
\thanks{\hskip-12ptThe second author was partially supported by Polish MNiSW 
grant N N201 541738}
\author{Andrzej Derdzinski \and \'Swiatos\l aw R.\ Gal}

\authorrunning{A.\ Derdzinski \and \'S.\ R.\ Gal} 

\institute{A.\ Derdzinski \at
Department of Mathematics, The Ohio State University, Columbus, OH 43210, USA\\
              Tel.: +1-614-292-4012\\
              Fax: +1-614-292-1479\\ 
\email{andrzej@math.ohio-state.edu}
\and
\'S.\ R.\ Gal \at 
Mathematical Institute, Wroc\l aw University, pl.\ Grunwaldzki 2/4, 50-384 
Wroc\l aw, Poland\\
\email{Swiatoslaw.Gal@math.uni.wroc.pl}
}

\date{\ }

\maketitle 

\begin{abstract}
The Killing form $\beta$ of a real (or complex) semisimple Lie group $G$ is a 
left-invariant pseudo-Riem\-annian (or, respectively, holomorphic) Einstein 
metric. Let $\Omega$ denote the multiple of its curvature operator, acting on 
symmetric 2-tensors, with the factor chosen so that $\Omega\beta=2\beta$. The 
result of Meyberg \cite{meyberg}, describing the spectrum of $\Omega$ in 
complex simple Lie groups $G$, easily implies that $1$ is not an eigenvalue of 
$\Omega$ in any real or complex simple Lie group $G$ except those locally 
isomorphic to SU($p,q$), or SL($n,\bbR$), or SL($n,\bbC$) or, for even $n$ 
only, SL($n/2,\bbH$), where $p\ge q\ge0$ and $\,p+q=n>2$. Due to the last 
conclusion, on simple Lie groups $G$ other the ones just listed, nonzero 
multiples of the Killing form $\beta$ are isolated among left-invariant 
Einstein metrics. Meyberg's theorem also allows us to understand the kernel of 
$\Lambda$, which is another natural operator. This in turn leads to a proof of 
a known, yet unpublished, fact: namely, that a sem\-i\-sim\-ple real or 
complex Lie algebra with no simple ideals of dimension 3 is essentially 
determined by its Cartan three-form.
\keywords{Simple Lie group \and indefinite Ein\-stein metric \and 
left-in\-var\-i\-ant Ein\-stein metric \and Cartan three-form}
\subclass{17B20 \and 22E46 \and 53C30}
\end{abstract} 

\setcounter{section}{0}
\renewcommand{\thethm}{\Alph{thm}}
\section{Introduction}\label{in}
\setcounter{equation}{0}
Every real  Lie group $\,\gj\,$ carries a distinguished left-in\-var\-i\-ant 
tor\-sion\-free connection $\,D$, defined by $\,D\hn_\x\w\hs\y=[\x,\y]/2\,$ 
for all left-in\-var\-i\-ant vector fields $\,\x\,$ and $\,\y$. In view of the 
Ja\-co\-bi identity, the curvature tensor of $\,D\,$ is $\,D\hh$-par\-al\-lel, 
and hence so is the Ric\-ci tensor of $\,D$, equal to a nonzero multiple of 
the Kil\-ling form $\,\by$. Our convention about $\,\by\,$ reads
\begin{equation}\label{bvw}
\by(\x,\x)\,=\,\mathrm{tr}\hskip1.6pt[\mathrm{Ad}\,x]^2\hskip12pt\mathrm{for\ any}
\hskip6ptx\hskip6pt\mathrm{in\ the\ Lie\ algebra}\hskip6pt\fg\hskip6pt
\mathrm{of\ }\,\hs\gj\hh.
\end{equation}
Thus, if $\,\gj\,$ is sem\-i\-sim\-ple, $\,\by\,$ constitutes a 
bi-in\-var\-i\-ant, locally symmetric, non-Ric\-ci-flat 
pseu\-\hbox{do\hskip.4pt-}\hskip0ptRiem\-ann\-i\-an Ein\-stein metric on 
$\,\gj$, with the Le\-vi-Ci\-vi\-ta connection $\,D$. We denote by 
$\,\cu:[\fg\nh^*]^{\odot2}\nnh\to[\fg\nh^*]^{\odot2}$ a specific multiple of 
the curvature operator of the metric $\,\by$, acting on symmetric 
symmetric bi\-lin\-e\-ar forms $\,\sy:\fg\to\fg$, so that, whenever 
$\,\x,\y\in\fg$,
\begin{equation}\label{oms}
\begin{array}{l}
\mathrm{a)}\hskip5pt[\hs\cu\nh\sy](\x,\y)\,
=\,2\,\mathrm{tr}\hskip1.6pt[(\mathrm{Ad}\,\x)\hh
(\mathrm{Ad}\,\y)\hh\Sigma]\hh,
\hskip15pt\mathrm{for}\hskip5pt\Sigma:\fg\to\fg\hskip5pt\mathrm{with}\\
\mathrm{b)}\hskip5pt\sy(\x,\y)\,=\,\by(\Sigma \x,\y)\hs.
\end{array}
\end{equation}
See Remark~\ref{crvop}. The same formula (\ref{oms}) defines the operator 
$\,\cu\,$ in a {\em complex\/} sem\-i\-sim\-ple Lie group $\,\gj$, acting on 
symmetric com\-plex-bi\-lin\-e\-ar forms $\,\sy$. We then identify $\,\cu\,$ 
with the analogous curvature operator for the ($\bbC$-bi\-lin\-e\-ar) 
Kil\-ling form $\,\by$, treating the latter as a hol\-o\-mor\-phic Ein\-stein 
metric on the underlying complex manifold of $\,\gj$.

The structure of $\,\cu\,$ in complex simple Lie groups is known from the 
work of Meyberg \cite{meyberg}, who showed that $\,\cu\,$ is 
di\-ag\-o\-nal\-izable and described its spectrum. For the reader's 
convenience, we reproduce Meyberg's theorem in an appendix. His result easily 
leads to a similar description of the spectrum of $\,\cu\,$ in real simple Lie 
algebras $\,\fg$, which we state as Theorem~\ref{rspec} and derive in 
Section~\ref{so} from the fact that, given any such $\,\fg$,
\begin{equation}\label{rcp}
\begin{array}{l}
\mathrm{a)}\hskip5pt
\mathrm{either\ }\,\fg\,\mathrm{\ is\ a\ real\ form\ of\ a\ complex\ 
simple\ Lie\ algebra\ }\,\fh\hh,\mathrm{\ or}\\
\mathrm{b)}\hskip5pt
\fg\,\mathrm{\ arises\ by\ treating\ a\ complex\ simple\ Lie\ algebra\ 
}\,\fh\,\mathrm{\ as\ real.}
\end{array}
\end{equation}
See \cite[Lemma 4 on p.\ 173]{hausner-schwartz}. The 
\hbox{Lie-}\hh al\-ge\-bra iso\-mor\-phism types of real simple Lie algebras 
$\,\fg\,$ thus form two disjoint classes, characterized by (\ref{rcp}.a) and 
(\ref{rcp}.b).

For both real and complex sem\-i\-sim\-ple Lie groups $\,\gj$, studying 
$\,\cu\,$ can be further motivated as follows. Let `metrics' on $\,\gj\,$ be, 
by definition, pseu\-\hbox{do\hskip.4pt-}\hskip0ptRiem\-ann\-i\-an or, 
respectively, hol\-o\-mor\-phic, and $\,\ee\,$ denote the set of 
Le\-vi-Ci\-vi\-ta connections of left-in\-var\-i\-ant Ein\-stein metrics on 
$\,\gj$. Then, as shown in \cite[Remark 12.3]{derdzinski-gal}, whenever a 
sem\-i\-sim\-ple Lie group $\,\gj\,$ has the property that $\,1\,$ is {\em 
not\/} an eigenvalue of $\,\cu$, the Le\-vi-Ci\-vi\-ta connection 
$\,D\,$ of its Kil\-ling form $\,\by\,$ is an isolated point of 
$\,\ee$. The converse implication holds except when $\,\gj\,$ is locally 
isomorphic to $\,\mathrm{SU}\hh(n)$, with $\,n\ge3$. See 
\cite[Theorems 22.2 and 22.3]{derdzinski-gal}.

In a real\hs$/$com\-plex Lie algebra $\,\fg$, we define 
$\,\m:[\fg\nh^*]^{\odot2}\nnh\to[\fg\nh^*]^{\wedge4}$ by
\begin{equation}\label{lgg}
(\m\hskip-1.1pt\sy)(\x,\y,\z,\z\hh'\hh)=\sy([\x,\y],[\z,\z\hh'\hh])
+\sy([\y,\z],[\x,\z\hh'\hh])+\sy([\z,\x],[\y,\z\hh'\hh])\hh. 
\end{equation}
Thus, $\,\m\,$ is a real\hs$/$com\-plex-lin\-e\-ar operator, sending symmetric 
bi\-lin\-e\-ar forms $\,\sy\,$ on $\,\fg\,$ to exterior $\,4$-forms on 
$\,\fg$. For the Kil\-ling form $\,\by\,$ one has 
$\,\by([\x,\y],[\z,\z\hh'\hh])=\by([[\x,\y],\z],\z\hh'\hh)$, as 
$\,\mathrm{Ad}\,\z\,$ is $\,\by$-skew-ad\-joint. Furthermore, by the 
Ja\-co\-bi identity and (\ref{bvw}) -- (\ref{oms}.a),
\begin{equation}\label{lbz}
\mathrm{i)}\hskip8pt\m\nh\by\,=\,0\hh,\hskip29pt
\mathrm{ii)}\hskip8pt\cu\by\,=\,2\by\hh.
\end{equation}
If, in addition, $\,\fg\,$ is sem\-i\-sim\-ple, there is also the operator 
$\,\Pi:[\fg\nh^*]^{\otimes4}\nnh\to[\fg\nh^*]^{\otimes2}$ with
\begin{equation}\label{pxp}
\Pi(\xy\otimes\xy\hh'\nnh\otimes\eta\otimes\eta'\hh)\,\,
=\,\,\by([x,x\hh'\hh],\,\cdot\,)
\hs\otimes\hs\by([y,y\hh'\hh],\,\cdot\,)\hh,
\end{equation}
for $\,\xy,\xy\hh'\nnh,\eta,\eta'\nh\in\fg\nh^*\nh$, where 
$\,x,x\hh'\nnh,y,y\hh'\nh\in\fg\,$ are characterized by 
$\,\xy=\by(x,\,\cdot\,),\,\xy\hh'\nh=\by(x\hh'\nnh,\,\cdot\,),\,
\eta=\by(y,\,\cdot\,),\,\eta\hh'\nh=\by(y\hh'\nnh,\,\cdot\,)$. Formula 
(\ref{osj}) below shows that 
$\,\Pi([\fg\nh^*]^{\wedge4})\subset[\fg\nh^*]^{\odot2}\nnh$.

Our first main result, established in Section~\ref{cd}, relates $\,\cu\,$ to 
$\,\Pi\m:[\fg\nh^*]^{\odot2}\nnh\to[\fg\nh^*]^{\odot2}\nnh$, the composite 
of $\,\m\,$ and the restriction of $\,\Pi\,$ to the sub\-space 
$\,[\fg\nh^*]^{\wedge4}\nh\subset[\fg\nh^*]^{\otimes4}\nnh$.
\begin{thm}\label{hstrh}
Let\/ $\,\cu,\,\m\,$ and\/ $\,\Pi\,$ be the operators defined by\/ 
{\rm(\ref{oms})}, {\rm(\ref{lgg})} and\/ {\rm(\ref{pxp})} \hskip1.5ptfor a 
given sem\-i\-sim\-ple real$/\hn$com\-plex Lie algebra\/ $\,\fg\hh$. Then\/ 
$\,\,2\hs\Pi\m=-\hh(\cu+\mathrm{Id})(\cu-2\hskip1.4pt\mathrm{Id})$.
\end{thm}
Next, in Section~\ref{pc}, we use Meyberg's result and Theorem~\ref{hstrh} to 
obtain the following description of $\,\mathrm{Ker}\hskip2.7pt\m\,$ for 
sem\-i\-sim\-ple Lie algebras $\,\fg$. It provides a crucial step in our proof 
of Theorem~\ref{autcf} (see below).
\begin{thm}\label{solut}
Given a real$/\hn$com\-plex sem\-i\-sim\-ple Lie 
algebra\/ $\,\fg\,$ with a di\-rect-sum decomposition\/ 
$\,\fg=\fy_1\w\nh\oplus\ldots\oplus\fy_s\w$ into 
simple ideals, $\,s\ge1$, let\/ $\,\m\,$ and\/ $\,\m_i\w$ denote the 
operator defined by\/ {\rm(\ref{lgg})} \hskip2ptfor\/ $\,\fg\,$ and, 
respectively, its analog for the\/ $\,i\hs$th summand\/ $\,\fyi\w\hh$.
\begin{enumerate}
  \def\theenumi{{\rm\roman{enumi}}}
\item[{\rm(i)}] $\mathrm{Ker}\hskip2.7pt\m\,=\,\mathrm{Ker}\hskip2.7pt\m_1\w\hs\oplus\ldots
\oplus\,\mathrm{Ker}\hskip2.7pt\m_s\w\hs$, where\/ 
$\,[\fyi^{\nh*}]^{\odot2}\nnh\subset[\fg^{\nh*}]^{\odot2}$ via trivial 
extensions.
\item[{\rm(ii)}] $\m=\hs0\,\,\,$ if\/ $\,\dim\fg=3\hh$.
\item[{\rm(iii)}] $\dim\,\mathrm{Ker}\hskip2.7pt\m=12\,\,\,$ if\/ 
$\,\fg\,$ is simple and\/ $\,\dim\fg=6$, which happens only when\/ $\,\fg\,$ 
is real and iso\-mor\-phic to the underlying real Lie algebra of\/ 
$\,\fsl\hh(2,\bbC)$, while\/ $\,\mathrm{Ker}\hskip2.7pt\m\,$ then consists of 
the real parts of all symmetric\/ $\,\bbC$-bi\-lin\-e\-ar functions\/ 
$\,\fg\times\fg\to\bbC\hh$.
\item[{\rm(iv)}] $\dim\,\mathrm{Ker}\hskip2.7pt\m\hs\in\hs\{1,2\}\,\,$ 
whenever\/ $\,\fg\,$ is simple and\/ $\,\dim\fg\notin\{3,6\}$, while\/ 
$\,\mathrm{Ker}\hskip2.7pt\m\,$ is then spanned either by 
the Kil\-ling form\/ $\,\by$, or by\/ $\,\mathrm{Re}\,\byh\,$ and\/ 
$\,\mathrm{Im}\,\byh\nh$. The former case occurs if\/ $\,\fg\,$ is complex, or 
real of type\/ {\rm(\ref{rcp}.a)}, the latter if\/ $\,\fg\,$ is real of type\/ 
{\rm(\ref{rcp}.b)}, with\/ $\,\byh$ denoting the Kil\-ling form of a complex 
simple Lie algebra\/ $\,\fh\,$ in\/ {\rm(\ref{rcp}.b)}.
\end{enumerate}
\end{thm}
Finally, one defines the {\em Car\-tan\/} {\em three-form\/} 
$\,\cf\in[\fg\nh^*]^{\wedge3}$ of a Lie algebra $\,\fg\,$ by
\begin{equation}\label{oeb}
\cf=\by([\,\cdot\,,\,\cdot\,],\,\cdot\,)\hh,\hskip12pt\mathrm{where}
\hskip7pt\by\hskip7pt\mathrm{denotes\ the\ Kil\-ling\ form.}
\end{equation}
The following result has been known for decades, although no published proof 
of it seems to exist \cite{bryant}. By an {\em isomorphism of the Car\-tan\/} 
{\em three-forms\/} we mean here a vec\-tor-space isomorphism of the Lie 
algebras in question, sending one \hbox{three\hh-}\hskip0ptform onto the other.
\begin{thm}\label{autcf}Let\/ $\,\fg\,$ be a real$/\hn$com\-plex 
sem\-i\-sim\-ple Lie algebra with a fixed di\-rect-sum decomposition into 
simple ideals, which we briefly refer to as the ``summands'' of\/ $\,\fg\hs$.
\begin{enumerate}
  \def\theenumi{{\rm\alph{enumi}}}
\item[{\rm(i)}] If\/ $\,\fh\,$ is a real$/\hn$com\-plex Lie algebra, the 
Car\-tan\/ \hbox{three-}\hskip0ptforms of\/ $\,\fg\,$ and\/ $\,\fh\,$ are 
isomorphic and, in the real case, $\,\fg\,$ has no summands of dimension\/ 
$\,3$, then\/ $\,\fh\,$ is isomorphic to\/ $\,\fg\hs$.
\item[{\rm(ii)}] If\/ $\,\fg\,$ contains no summands of dimension\/ $\,3\,$ 
or\/ $\,6$, then every auto\-mor\-phism of the Cartan\/ 
\hbox{three\hh-}\hskip0ptform of\/ $\,\fg\,$ is a Lie-\hh al\-ge\-bra 
auto\-mor\-phism of\/ $\,\fg\,\,$ followed by an operator that acts on each 
summand as the multiplication by a cubic root of\/ $\,1$.
\item[{\rm(iii)}] If\/ $\,\fg\,$ is the underlying real Lie algebra of a 
complex simple Lie algebra and\/ $\,\dim\fg\ne6\hh$, then every 
auto\-mor\-phism of the Cartan\/ \hbox{three\hh-}\hskip0ptform of\/ $\,\fg\,$ 
is com\-plex-lin\-e\-ar or anti\-lin\-e\-ar.
\end{enumerate}
Conversely, if\/ $\,\fg\,$ has\/ $\,\,k\,$ summands of dimension\/ $\,3\,$ 
and\/ $\,l\,$ summands of dimension\/ $\,6\hh$, then the Lie-\hh al\-ge\-bra 
auto\-mor\-phisms of\/ $\,\fg\,\,$ form a subgroup of co\-di\-men\-sion\/ 
$\,\,5\hh k+12\hs l\,$ in the auto\-mor\-phism group of the Cartan\/ 
\hbox{three\hh-}\hskip0ptform.
\end{thm}
We derive Theorem~\ref{autcf} from Theorem~\ref{solut}, in Section~\ref{pa}. 

\renewcommand{\thethm}{\thesection.\arabic{thm}}
\section{Preliminaries}\label{pr}
\setcounter{equation}{0}\setcounter{thm}{0}
Suppose that $\,\fg\,$ is the underlying real Lie algebra of a complex Lie 
algebra $\,\fh$. We denote by $\,\by\,$ and $\,\cf\,$ the Kil\-ling form and 
Car\-tan \hbox{three\hh-}\hskip0ptform of $\,\fg$, cf.\ (\ref{bvw}) and 
(\ref{oeb}), by $\,\m\,$ the operator in (\ref{lgg}) associated with $\,\fg$, 
and use the symbols $\,\byh\nnh,\hs\cf\hh^\fh\nnh,\hs\m\hskip-1.5pt^\fh$ for 
their counterparts corresponding to $\,\fh$. Obviously, whenever $\,\sy:\fg\times\fg\to\bbC\,$ is 
a symmetric $\,\bbC\hn$-bi\-lin\-e\-ar form,
\begin{equation}\label{lrs}
\mathrm{i)}\hskip5pt
\by=2\,\mathrm{Re}\,\byh,\hskip12pt
\mathrm{ii)}\hskip5pt
\cf=2\,\mathrm{Re}\,\cf\hh^\fh,\hskip12pt
\mathrm{iii)}\hskip5pt\m(\mathrm{Re}\,\sy)
=\mathrm{Re}\,(\m\hskip-1.5pt^\fh\sy)\hh.
\end{equation}
For (\ref{lrs}.i), see also \cite[formula\ (13.1)]{derdzinski-gal}. With 
$\,\fg\,$ and $\,\fh\,$ as above, it is clear from (\ref{lrs}.i) that
\begin{equation}\label{und}
\begin{array}{l}
\mathrm{Re}\,\byh\mathrm{\ and\ }\,\mathrm{Im}\,\byh\mathrm{\ span\ the\ real\ 
space\ of\ symmetric\ bi\-lin\-e\-ar\ forms\ }\,\sy\,\mathrm{\ on\ }\,\fg\,
\mathrm{\ arising}\\
\mathrm{via\ (\ref{oms}.b)\ from\ linear\ en\-do\-mor\-phisms\ }\,\,\Sigma\,\,
\mathrm{\ which\ are\ complex\ multiples\ of\ }\,\,\mathrm{Id}\hh.
\end{array}
\end{equation}
Furthermore, (\ref{lrs}.i) also implies, for dimensional reasons, that
\begin{equation}\label{cxl}
\begin{array}{l}
\mathrm{the\ real\ parts\ of\ symmetric\ }\,\bbC\hn\hyp\mathrm{bi\-lin\-e\-ar\ 
functions\ }\,\fg\times\fg\to\bbC\,\mathrm{\ form\ the\ image}\\
\mathrm{under\ (\ref{oms}.b)\ of\ the\ space\ of\ 
}\,\bbC\hyp\mathrm{lin\-e\-ar\ 
}\,\byh\nnh\hyp\mathrm{self}\hyp\mathrm{ad\-joint\ en\-do\-mor\-phisms\ of\ }
\,\fh\hh,
\end{array}
\end{equation}
as the former space obviously contains the latter.

Let $\,\fg\,$ now be a Lie algebra over the scalar field $\,\bbF=\bbR\,$ or 
$\,\bbF=\bbC$. A fixed basis of $\,\fg$ allows us to represent elements 
$\,x,y\,$ of $\,\fg$, symmetric bi\-lin\-e\-ar forms $\,\sy\,$ on $\,\fg$, and 
the \hbox{Lie\hh}-\hskip0ptal\-ge\-bra bracket operation 
$\,[\hskip1.6pt,\hskip.6pt]\,$ by their components 
$\,x^{\hh i}\nnh,y^{\hh i}\nnh$, $\,\sy_{ij}\w$ and $\,C_{ij}\w{}^k$ 
(the structure constants of $\,\fg$), so that 
$\,\sy(x,y)=\sy_{ij}\w\hs x^{\hh i}y^{\hs j}$ and 
$\,[x,y]^k\nh=C_{ij}\w{}^kx^{\hs i}y^{\hs j}$. Repeated indices are summed 
over. The Car\-tan \hbox{three\hh-}\hskip0ptform $\,\cf\,$ with (\ref{oeb}) 
has the components $\,C_{ijk}\w=C_{ij}\w{}^r\by_{kr}\w$, where $\,\by\,$ is 
the Kil\-ling form. The definition (\ref{bvw}) of $\,\by$, its 
bi-in\-var\-i\-ance, and the Ja\-co\-bi identity now read
\begin{equation}\label{bij}
\begin{array}{l}
\phantom{\mathrm{ii}}\mathrm{i)}\hskip8pt\by_{ij}\w
=C_{ip}\w{}^qC_{\!jq}\w{}^p,\phantom{_{j_{j_j}}}\hskip28pt
\mathrm{ii)}\hskip8ptC_{ijk}\w\hskip5pt\mathrm{is\ 
skew}\hyp\mathrm{sym\-met\-ric\ in}\hskip6pti,j,k\hh,\\
\mathrm{iii)}\hskip8ptC_{ij}\w{}^qC_{\nh qk}\w{}^l\hs
+\,C_{\nh jk}\w{}^qC_{\nh qi}\w{}^l\hs
+\,C_{\hn ki}\w{}^qC_{\nh qj}\w{}^l\hs=\,0\hh.
\end{array}
\end{equation}
In the remainder of this section $\,\fg\,$ is also assumed to be 
sem\-i\-sim\-ple. We can thus lower and raise indices using the components 
$\,\by_{ij}\w$ of the Kil\-ling form $\,\by\,$ and $\,\by^{ij}$ of its 
reciprocal: $\,C^{\hh k}{}\nnh_p\w{}^q\nh=\by^{kr}C_{\hn rp}\w{}^q$, and 
$\,C_{\!j}\w{}^{sp}=\by^{sk}C_{\!jk}\w{}^p\nnh$. For any 
$\,\x,\y,\z\in\fg$, one has  
$\,2\,\mathrm{tr}\hskip1.6pt[(\mathrm{Ad}\,\x)\hh(\mathrm{Ad}\,\y)\hh
(\mathrm{Ad}\,\z)]=\cf(\x,\y,\z)$, where $\,\cf\,$ is the Car\-tan 
\hbox{three\hh-}\hskip0ptform given by (\ref{oeb}). Equivalently,
\begin{equation}\label{xyz}
2\hh C_{ir}\w{}^pC_{\!jq}\w{}^rC_{\hn kp}\w{}^q\,=\,\,C_{ijk}\w\hh.
\end{equation}
In fact, by successively using the equalities 
$\,C^{\hh k}{}\nnh_p\w{}^q\nh=C_{\nh p}\w{}^{qk}$ and 
$\,C_{i}\w{}^{r\nh p}\nh=-\hh C_{i}\w{}^{pr}$ (both due to 
(\ref{bij}.ii)), then again (\ref{bij}.ii), (\ref{bij}.iii), and 
(\ref{bij}.i\hh--\hh ii), we get 
$\,2\hh C_{ir}\w{}^pC_{\!jq}\w{}^rC^{\hh k}{}\nnh_p\w{}^q\nh
=2\hh C_{i}\w{}^{r\nh p}C_{\!jqr}\w C_{\nh p}\w{}^{qk}\nh
=\hs C_{i}\w{}^{r\nh p}(C_{\!jqr}\w C_{\nh p}\w{}^{qk}\nnh
-C_{\!jqp}\w C_{\!r}\w{}^{qk})
=C_{i}\w{}^{r\nh p}(C_{\!jr}\w{}^q C_{\nh qp}\w{}^k\nnh
+C_{\nh pj}\w{}^q C_{\nh qr}\w{}^k)
=-\hh C_{i}\w{}^{r\nh p}C_{\hn rp}\w{}^q C_{\nh qj}\w{}^k\nh
=\vd_i^q C_{\nh qj}\w{}^k\nh=C_{ij}\w{}^k\nnh$. Lowering the index $\,k$, we 
obtain (\ref{xyz}). Next, we introduce the linear operator
\begin{equation}\label{zee}
\tj:[\fg\nh^*]^{\otimes2}\nnh\to[\fg\nh^*]^{\otimes2}\hskip6pt\mathrm{with}
\hskip5pt(\tj\nnh\sy)_{ij}\w
=\tj_{ij}^{\hs kl}\sy_{kl}\w\hh,\hskip4pt\mathrm{where}
\hskip5pt\tj_{ij}^{\hs kl}=2\hh C_{ip}\w{}^kC_{\!j}\w{}^{lp}.
\end{equation}
\begin{lem}\label{tndom}For\/ $\,\tj\,$ and the operator\/ 
$\,\cu:[\fg\nh^*]^{\odot2}\nnh\to[\fg\nh^*]^{\odot2}$ given by\/ 
{\rm(\ref{oms})},
\begin{enumerate}
  \def\theenumi{{\rm\alph{enumi}}}
\item[{\rm(a)}] $\tj\,$ leaves the subspaces\/ $\,[\fg\nh^*]^{\odot2}$ and\/ 
$\,[\fg\nh^*]^{\wedge2}$ invariant,
\item[{\rm(b)}] $\cu\,$ coincides with the restriction of\/ $\,\tj\,$ to\/ 
$\,[\fg\nh^*]^{\odot2}\nnh$.
\end{enumerate}
\end{lem}
\begin{proof}Our claim is obvious from (\ref{zee}) and the fact that, by 
(\ref{zee}), $\,\tj\nnh\sy\,$ is the same as $\,\cu\nh\sy\,$ in (\ref{oms}), 
except that now $\,\sy:\fg\times\fg\to\bbF\,$ need not be symmetric.\qed
\end{proof}
\begin{lem}\label{eigen}For any complex simple Lie algebra\/ $\,\fg\,$ and\/ 
$\,\cu:[\fg\nh^*]^{\odot2}\nnh\to[\fg\nh^*]^{\odot2}$ with\/ {\rm(\ref{oms})},
\begin{enumerate}
  \def\theenumi{{\rm\alph{enumi}}}
\item[{\rm(a)}] $\cu\,$ is di\-ag\-o\-nal\-izable,
\item[{\rm(b)}] $2\,$ is an eigen\-val\-ue of\/ $\,\cu\,$ with multiplicity\/ 
$\,1$,
\item[{\rm(c)}] $0\,$ is not an eigen\-value of\/ $\,\cu^\fh\nh$,
\item[{\rm(d)}] $\cu\,$ has the eigen\-value\/ $\,1\,$ if and only if\/ 
$\,\fg\,$ is iso\-mor\-phic to\/ $\,\fsl\hh(\n,\bbC)\,$ for some\/ $\,\n\ge3\hh$.
\end{enumerate}
\end{lem}
\begin{proof}This is a special case of Meyberg's theorem, stated in the 
Appendix.\qed
\end{proof}
\begin{rem}\label{clssf}{\rm 
The iso\-mor\-phism types of all complex simple Lie algebras are: 
$\,\fsl_n\w$, for $\,n\ge2$, $\,\fsp_n\w$ (even $\,n\ge4$), $\,\fso_n\w$ with 
$\,n\ge7$, as well as 
$\,\fg_2\w,\hs\mathfrak{f}_4\w,\hs\mathfrak{e}_6\w,\hs\mathfrak{e}_7\w$ and 
$\,\hs\mathfrak{e}_8\w$. See \cite[pp.\ 8 and 77]{onishchik}.
}\end{rem}
\begin{rem}\label{crvop}{\rm 
The curvature operator of a 
(pseu\-do)\hskip0ptRiem\-ann\-i\-an metric $\,\gy\,$ on a manifold, acting on 
symmetric $\,2$-ten\-sors, has been studied by various authors  
\cite{calabi-vesentini}, \cite{bourguignon-karcher}, \cite[pp.\ 
51--52]{besse}. It is given by $\,\sy\mapsto\ty$, where 
$\,2\ty_{ij}\w=\gy\hs^{pq}\rj_{ipj}\w{}^k\sy_{\nh qk}\w$ in terms of 
components relative to a basis of the tangent space at any point, the sign 
convention about the curvature tensor $\,\rj\,$ being that a Euclidean tangent 
plane with an orthonormal basis $\,\x,\y\,$ has the sectional curvature 
$\,\gy_{pq}\w\rj_{ijk}\w{}^p\x^{\hh i}\y^{\hs j}\x^{\hh k}\y^q\nh$. When 
$\,\gy\,$ is the Kil\-ling form $\,\by\,$ of a sem\-i\-sim\-ple Lie group 
$\,\gj$, treated as a left-in\-var\-i\-ant metric (see the lines following 
(\ref{bvw})), this operator equals $\,-\hh\cu/\hn16$, for $\,\cu$ 
with (\ref{oms}). In fact, the description of the Le\-vi-Ci\-vi\-ta connection 
$\,D\,$ of $\,\by\,$ in the Introduction gives 
$\,4\hh\rj(\x,\y)\hh\z=[[\x,\y],\z]\,$ for left-in\-var\-i\-ant vector fields 
$\,\x,\y,\z$, that is, 
$\,4\hh\rj_{ijk}\w{}^l\nh=C_{ij}\w{}^pC_{\nh pk}\w{}^l\nnh$. 
Lemma~\ref{tndom}(b) now implies our claim, as 
$\,\tj_{ij}^{\hs kl}=-8\by^{kp}\rj_{\hskip-1.1ptjpi}\w{}^l$ due to 
(\ref{bij}.ii) and (\ref{zee}).
}\end{rem}

\section{Proof of Theorem~\ref{hstrh}}\label{cd}
\setcounter{equation}{0}\setcounter{thm}{0}
We use the component notation of Section~\ref{pr}. According to (\ref{lgg}) 
and (\ref{pxp}),
\begin{equation}\label{osj}
\begin{array}{l}
(\m\hskip-1.1pt\sy)_{ijkl}\w
=\m_{ijkl}\w{}^{rs}\sy_{\hn rs}\w\hskip8pt\mathrm{with}
\hskip6pt\m_{ijkl}\w{}^{rs}\nh=\hs C_{ij}\w{}^rC_{\hn kl}\w{}^s\nh
+C_{\nnh jk}\w{}^rC_{il}\w{}^s\nh+C_{\hn ki}\w{}^rC_{\nnh jl}\w{}^s\nh,\\
(\Pi\hh\zy)_{\nh pq}\w
=\hs C^{\hs ij}{}\nnh_p\w C^{\hs kl}{}\nnh_q\w\zy_{\hh ijkl}\w\hh,\hskip9pt
\mathrm{whenever}\hskip5pt\sy\in[\fg\nh^*]^{\odot2}\hskip4pt\mathrm{and}
\hskip5pt\zy\in[\fg\nh^*]^{\wedge4}\nh.
\end{array}
\end{equation}
In any real\hs$/$com\-plex sem\-i\-sim\-ple Lie algebra $\,\fg$, for 
$\,C_{ij}\w{}^k\nnh,\hs\tj_{ij}^{\hs kl}$ as in Section~\ref{pr},
\begin{equation}\label{ccc}
2C^{\hs ij}{}\nnh_p\w C^{\hs kl}{}\nnh_q\w(C_{ij}\w{}^rC_{\hn kl}\w{}^s\nh
+C_{\nnh jk}\w{}^rC_{il}\w{}^s\nh+C_{\hn ki}\w{}^rC_{\nnh jl}\w{}^s)\hs
=\,2\hh\vd_p^r\vd_q^s\,+\,\tj_{\nh pq}^{\hs rs}\,
-\,\tj_{\nh pq}^{\hs ik}\tj_{\nh ik}^{\hs rs}\hh.
\end{equation}
In fact, the first of the three terms naturally arising on the left-hand side 
of (\ref{ccc}) equals $\,2\hh\vd_p^r\vd_q^s$ since, by 
(\ref{bij}.i\hh--\hh ii), 
$\,C^{\hs ij}{}\nnh_p\w C_{ij}\w{}^r\nh=-\hh\vd_p^r$ and 
$\,C^{\hs kl}{}\nnh_q\w C_{\hn kl}\w{}^s\nh=-\hh\vd_q^s$. The other two terms 
coincide (as skew-sym\-me\-try of $\,C^{\hs ij}{}\nnh_p\w$ in $\,i,j\,$ gives 
$\,C^{\hs ij}{}\nnh_p\w C_{\nnh jk}\w{}^rC_{il}\w{}^s\nh
=-\hh C^{\hs ij}{}\nnh_p\w C_{\nnh ik}\w{}^rC_{\nh jl}\w{}^s\nh
=C^{\hs ij}{}\nnh_p\w C_{\hn ki}\w{}^rC_{\nnh jl}\w{}^s$), and so they add up to 
$\,4\hh C^{\hs ij}{}\nnh_p\w C^{\hs kl}{}\nnh_q\w 
C_{\hn ki}\w{}^rC_{\nnh jl}\w{}^s\nnh$, that is, 
$\,4\hh C^{\hs kl}{}\nnh_q\w C_{\nnh jl}\w{}^s
C_{\nh p}\w{}^{ji}C_{ik}\w{}^r\nh
=4\hh C^{\hs kl}{}\nnh_q\w C^{\hh j}{}_l\w{}^sC_{\nh pj}\w{}^iC_{ik}\w{}^r\nh
=-\hh4\hh C^{\hs kl}{}\nnh_q\w C^{\hh j}{}_l\w{}^s
(C_{\nnh jk}\w{}^iC_{ip}\w{}^r\nh
+C_{\hn kp}\w{}^iC_{ij}\w{}^r)$; the rightmost equality is due to the 
Ja\-co\-bi identity (\ref{bij}.iii). The last expression consists of the 
first term, $\,-\hh4\hh C^{\hs kl}{}\nnh_q\w C^{\hh j}{}\nh_l\w{}^s
C_{\nnh jk}\w{}^iC_{ip}\w{}^r\nh
=-\hh4\hh C_{ip}\w{}^r(C^{\hs i}{}_k\w{}^j
C_{\!ql}\w{}^kC^{\hh s}{}\nnh_j\w{}^l)
=-\hh4\hh C_{ip}\w{}^rC^{\hs i}{}_q\w{}^s\nnh$, cf.\ (\ref{xyz}), equal, 
by (\ref{bij}.ii) and (\ref{zee}), to 
$\,C_{\nh pi}\w{}^rC_{\nh q}\w{}^{si}=\tj_{\nh pq}^{\hs rs}$, and the second 
term, $\,-\hh(2\hh C_{\hn kp}\w{}^iC^{\hs kl}{}\nnh_q\w)
(2\hh C_{ij}\w{}^rC^{\hh j}{}\nh_l\w{}^s)$, the two parenthesized factors of 
which are, for the same reasons, nothing else than 
$\,\tj_{\nh pq}^{\hs il}$ and $\,\tj_{\nh il}^{\hs rs}$. This proves (\ref{ccc}).

Theorem~\ref{hstrh} is now an obvious consequence of (\ref{osj}) -- 
(\ref{ccc}) and Lemma~\ref{tndom}(b).

\section{The spectrum of $\,\cu\,$ in real simple Lie algebras}\label{so}
\setcounter{equation}{0}\setcounter{thm}{0}
\begin{thm}\label{rspec}Let\/ $\,\cu\,$ denote the operator with\/ 
{\rm(\ref{oms})} corresponding to a fixed real simple Lie algebra\/ 
$\,\fg\hh$, and\/ $\,\cu^\fh$ its analog for\/ $\,\fh\hh$, chosen so that\/ 
$\,\fg\,$ and\/ $\,\fh\,$ satisfy\/ {\rm(\ref{rcp})}.
\begin{enumerate}
  \def\theenumi{{\rm\roman{enumi}}}
\item[{\rm(i)}] $\cu\,$ is always di\-ag\-o\-nal\-izable.
\item[{\rm(ii)}] In case\/ {\rm(\ref{rcp}.a)}, $\,\cu\,$ has the same spectrum 
as\/ $\,\cu^\fh\nh$, including the multiplicities.
\item[{\rm(iii)}] In case\/ {\rm(\ref{rcp}.b)}, the spectrum of\/ $\,\cu\,$ 
arises from that of\/ $\,\cu^\fh$ by doubling the original multiplicities and 
then including\/ $\,0\,$ as an additional eigen\-value with the required 
complementary multiplicity. Note that, by Lemma\/~{\rm\ref{eigen}(c)}, $\,0\,$ 
is not an eigen\-value of\/ $\,\cu^\fh\nh$.
\item[{\rm(iv)}] The eigen\-space\/ 
$\,\mathrm{Ker}\,(\cu-2\hskip1.4pt\mathrm{Id})\,$ is spanned in case\/ 
{\rm(\ref{rcp}.a)} by $\,\by$, and in case\/ {\rm(\ref{rcp}.b)} by\/ 
$\,\mathrm{Re}\,\byh\hh$ and\/ $\,\mathrm{Im}\,\byh\nh$, \hskip2ptfor the 
Kil\-ling forms\/ $\,\by\,$ of\/ $\,\fg\,$ and\/ $\,\byh$ of\/ $\,\fh\hh$.
\end{enumerate}
\end{thm}
\begin{proof}By \cite[Lemma 14.3(ii) and formulae (14.5) -- 
(14.7)]{derdzinski-gal}, if $\,\fg\,$ is of type (\ref{rcp}.a), the 
complexification of $\,[\fg\nh^*]^{\odot2}$ may be naturally identified with 
its (com\-plex) counterpart $\,[\fh\nh^*]^{\odot2}$ for $\,\fh$, in such a way 
that $\,\cu^\fh$ and the Kil\-ling form $\,\byh$ become the unique 
$\,\bbC$-lin\-e\-ar extensions of $\,\cu\,$ and $\,\by$. Now 
Lemma~\ref{eigen}(a)\hh--\hh(b) and (\ref{lbz}.ii) yield (i), (ii) and (iv) in 
case (\ref{rcp}.a).

For $\,\fg\,$ of type (\ref{rcp}.a), Lemma 13.1 of \cite{derdzinski-gal} 
states the following. First, 
$\,[\fg\nh^*]^{\odot2}$ is the direct sum of two $\,\cu$-in\-var\-i\-ant 
subspaces: one formed by the real parts of $\,\bbC\hn$-bi\-lin\-e\-ar 
symmetric functions $\,\sy:\fh\times\fh\to\bbC$, the other by the real parts 
of functions $\,\sy:\fh\times\fh\to\bbC\,$ which are anti\-lin\-e\-ar and 
Her\-mit\-i\-an. Secondly, $\,\cu\,$ vanishes on the ``Her\-mit\-i\-an'' 
summand, and its action on the ``symmetric'' summand is equivalent, via the 
iso\-mor\-phism $\,\sy\mapsto\mathrm{Re}\,\sy$, to the action of 
$\,\cu^\fh$ on $\,\bbC\hn$-bi\-lin\-e\-ar symmetric functions $\,\sy$. With 
di\-ag\-o\-nal\-iza\-bil\-i\-ty of $\,\cu^\fh$ again provided by 
Lemma~\ref{eigen}(a), this proves our remaining claims. (The multiplicities 
are doubled since the original complex eigen\-spaces are viewed as real, while 
the eigen\-space $\,\cu^\fh$ for the eigen\-val\-ue $\,2\,$ consists, by (v) 
and (\ref{lbz}.ii), of complex multiples of $\,\byh\nh$, the real parts of 
which are precisely the real linear combinations of $\,\mathrm{Re}\,\byh$ and 
$\,\mathrm{Im}\,\byh\nh$.)\qed
\end{proof}
\begin{rem}\label{realf}{\rm 
It is well known \cite[p.\ 30]{onishchik} that, up to iso\-mor\-phisms, 
$\,\fsl\hh(\n,\bbR)\,$ as well as $\,\fsu\hh(\p,\q)\,$ with $\,\p+\q=\n\,$ 
and, if $\,\n\,$ is even, $\,\fsl\hh(\n/2,\bbH)$, are the only real forms of 
$\,\fsl\hh(\n,\bbC)$.
}\end{rem}
\begin{lem}\label{thrdm}The only complex, or real, simple Lie algebras of 
dimensions less than\/ $\,7\,$ are, up to iso\-mor\-phisms, $\,\fsl\hh(2,\bbC)\,$ 
or, respectively, $\,\fsl\hh(2,\bbR)$, $\,\fsu\hh(2)$, $\,\fsu\hh(1,1)\,$ and\/ 
$\,\fsl\hh(2,\bbC)$, the last one being both complex 
\hbox{three\nh-}\hskip0ptdi\-men\-sion\-al and real 
\hbox{six\nh-}\hskip0ptdi\-men\-sion\-al. Consequently,
\begin{enumerate}
  \def\theenumi{{\rm\roman{enumi}}}
\item[{\rm(i)}] a complex simple Lie algebra cannot be 
\hbox{six\nh-}\hskip0ptdi\-men\-sion\-al,
\item[{\rm(ii)}] there is just one iso\-mor\-phism type of a complex or, 
respectively, real simple Lie algebra of dimension\/ $\,3\,$ or, respectively, 
$\,6\hh$, both represented by\/ $\,\fsl\hh(2,\bbC)$,
\item[{\rm(iii)}] $\,\dim\fg\notin\{1,2,4,5\}\,$ \hskip2ptfor every real or 
complex simple Lie algebra\/ $\,\fg\hh$.
\end{enumerate}
\end{lem}
\begin{proof}According to Remark~\ref{clssf}, in the complex case, only 
$\,\fsl\hh(2,\bbC)\,$ is possible. For real Lie algebras, one can use 
Remark~\ref{realf} and (\ref{rcp}).\qed
\end{proof}
\begin{rem}\label{oneig}{\rm 
We can now justify the claim, made in \cite[Remark 12.3]{derdzinski-gal}, that 
$\,1\,$ is not an eigen\-value of $\,\cu\,$ in any real or complex simple Lie 
algebra except the ones isomorphic to 
$\,\fsl\hh(\n,\bbR),\hs\fsl\hh(\n,\bbC),\hs\fsu\hh(\p,\q)\,$ or, for even 
$\,\n\,$ only, $\,\hs\fsl\hh(\n/2,\bbH)$, where $\,\n=\p+\q\ge3$.

In fact, by Theorem~\ref{eigen} and parts (ii) -- (iii) of 
Theorem~\ref{rspec}, the only real or complex simple Lie algebras in which 
$\,\cu\,$ has the eigen\-value $\,1\,$ are, up to iso\-mor\-phisms, 
$\,\fsl\hh(\n,\bbC)\,$ for $\,\n\ge3\,$ and their real forms. According to 
Remark~\ref{realf}, these are all listed in the last paragraph.
}\end{rem}

\section{Proof of Theorem~\ref{solut}}\label{pc}
\setcounter{equation}{0}\setcounter{thm}{0}
Let $\,\sy\in[\fg\nh^*]^{\odot2}$ and $\,\m\hskip-1.1pt\sy=0$. Consequently, 
by (\ref{lgg}), $\,\sy([\x,\y],[\z,\z\hh'\hh])+\sy([\y,\z],[\x,\z\hh'\hh])
+\sy([\z,\x],[\y,\z\hh'\hh])=0\,$ for all $\,\x,\y,\z,\z\hh'$ in $\,\fg$. 
Thus, $\,\sy([\x,\y],[\z,\z\hh'\hh])=0\,$ whenever $\,\x,\y\in\fh_i\w$ and 
$\,\z,\z\hh'\nh\in\fh\nh_j\w$ with $\,j\ne i$. The summands $\,\fh_i\w$ and 
$\,\fh\nh_j\w$, being simple, are spanned by such brackets $\,[\x,\y]\,$ and 
$\,[\z,\z\hh'\hh]$, and so $\,\fh_i\w$ is $\,\sy$-or\-thog\-o\-nal to 
$\,\fh\nh_j\w$. As this is the case for any two summands, we obtain (i), the 
right-to\hh-left inclusion being obvious. Next,
\begin{equation}\label{fbe}
\mathrm{Ker}\,(\cu-2\hskip1.4pt\mathrm{Id})\,
\subset\,\mathrm{Ker}\hskip2.7pt\m\,\subset\,\mathrm{Ker}\,(\cu-2\hskip1.4pt\mathrm{Id})
\oplus\mathrm{Ker}\,(\cu+\mathrm{Id})\hh.
\end{equation}
In fact, the second inclusion is obvious from Theorem~\ref{hstrh}; the first, 
from Theorem~\ref{rspec}(iv), (\ref{lbz}.i) and (\ref{lrs}.iii) applied to 
complex multiples $\,\sy\,$ of $\,\byh\nh$.

Part (ii) of Theorem~\ref{solut} is immediate, as 
$\,[\fg\nh^*]^{\wedge4}\nnh=\{0\}\,$ when $\,\dim\fg=3$. Also, if $\,\fg\,$ is 
simple and $\,\dim\fg=6$, Lemma~\ref{thrdm}(i)\hh-\hh(ii) implies that 
$\,\fg\,$ is real and iso\-mor\-phic to $\,\fsl\hh(2,\bbC)$. From 
(\ref{lrs}.iii), with $\,\m\hskip-1.5pt^\fh\sy=0\,$ by (ii), we now get 
$\,\ef\hs\subset\,\mathrm{Ker}\hskip2.7pt\m\,$ for 
$\,\ef=\{\mathrm{Re}\,\sy:\sy\in[\fg\nh^*]^{\odot2}\}$, where 
$\,[\fg\nh^*]^{\odot2}$ denotes the space of all symmetric 
$\,\bbC\hn$-bi\-lin\-e\-ar forms $\,\sy:\fg\times\fg\to\bbC$. As 
$\,\mathrm{Re}\,\sy\,$ uniquely determines such $\,\sy$, that is, the operator 
$\,\sy\mapsto\mathrm{Re}\,\sy\,$ is injective, we thus have $\,\dimr\ef=12$. 
The second inclusion in (\ref{fbe}) is therefore an equality, and 
$\,\ef=\mathrm{Ker}\hskip2.7pt\m$, for dimensional reasons: 
$\,\mathrm{Ker}\hskip2.7pt\m\,$ contains the subspace $\,\ef\,$ of real 
dimension $\,12$, equal, in view of part (a) of Theorem~\ref{eigen} and 
Theorem~\ref{rspec}(iii), to 
$\,\dimr\hs[\hs\mathrm{Ker}\,(\cu-2\hskip1.4pt\mathrm{Id})\nh
\oplus\nh\mathrm{Ker}\,(\cu+\mathrm{Id})\hh]$. This yields assertion (iii) in 
Theorem~\ref{solut}.

Let $\,\fg\,$ now be simple, with $\,\dim\fg\notin\{3,6\}$. Due to 
Theorems~\ref{eigen} and~\ref{rspec}(ii)\hh-\hh(iii), $\,-1\,$ is not an 
eigen\-val\-ue of $\,\cu$. Thus, $\,\mathrm{Ker}\,(\cu+\mathrm{Id})=\{0\}$, 
and the inclusions in (\ref{fbe}) are equalities. In view of 
Theorem~\ref{rspec}(iv), this completes the proof.

\section{Some needed facts from linear algebra}\label{sf}
\setcounter{equation}{0}\setcounter{thm}{0}
In this section $\,\fg\,$ is the underlying real space of a 
fi\-\hbox{nite\nh-}\hskip0ptdi\-men\-sion\-al complex vector space $\,\fh$ 
and $\,J:\fg\to\fg\,$ is the operator of multiplication by $\,i$, also 
referred to as the {\em complex structure}. We denote by $\,\byh$ a fixed 
nondegenerate $\,\bbC\hn$-bi\-lin\-e\-ar symmetric form on $\,\fh$, so that 
the $\,\bbR$-bi\-lin\-e\-ar symmetric form $\,\by=2\,\mathrm{Re}\,\byh$ on 
$\,\fg\,$ is nondegenerate as well. The same applies to any nonzero complex 
multiple of $\,\byh\nh$. Thus, $\,\by\,$ and $\,\gy=2\,\mathrm{Im}\,\byh$ 
constitute a basis of a real vector space $\,\ep\,$ of $\,\bbR$-bi\-lin\-e\-ar 
symmetric forms on $\,\fg$. All nonzero elements of $\,\ep\,$ are 
nondegenerate. As $\,\byh$ is $\,\bbC\hn$-bi\-lin\-e\-ar, 
$\,\gy(\x,y)=-\by(\x,J\y)\,$ for all $\,\x,\y\in\fg$. We use components 
relative to a basis of $\,\fg$, as in Section~\ref{pr}.
\begin{lem}\label{dtcpx}The real spaces\/ $\,\fg\,$ and\/ $\,\ep\,$ uniquely 
determine the pair\/ $\,(J,\byh)\,$ up to its replacement by\/ 
$\,(J,a\hn\byh\nh)\,$ or\/ $\,(-J,a\hskip.9pt\overline{\hskip-1.3pt\byh})$, 
with any\/ $\,a\in\bbC\smallsetminus\{0\}$.
\end{lem}
\begin{proof}For any basis $\,\ky,\ly\,$ of $\,\ep\nh$, replacing $\,\byh$ by 
a complex multiple, which leaves $\,\ep$ unchanged, we assume that 
$\,\ky=\by$. Thus, $\,\ly=u\by+v\gy$, where $\,u,v\in\bbR\,$ and $\,v\ne0$. 
Writing the equality $\,\gy=-\by(\,\cdot\,,\hs J\,\cdot\,)\,$ as 
$\,\gy_{rq}\w=-\by_{rs}\w\hs J_q^s$, and then using the reciprocal components 
$\,\ky^{\hh pr}\nh=\by^{\nh pr}\nnh$, we obtain 
$\,\ky^{\hh pr}\ly_{\hh rq}\w=\by^{\nh pr}(u\by_{rq}\w-v\by_{rs}\w\hs J_q^s)
=u\vd_q^p-vJ_q^p$. Now $\,\pm\hs J\,$ may be defined by declaring the matrix 
$\,J_q^p$ to be the traceless part of $\,\ky^{\hh pr}\ly_{\hh rq}\w$, 
normalized so that $\,J^2\nh=-\hh\mathrm{Id}$.

At the same time, fixing any $\,\ky\in\ep\smallsetminus\{0\}\,$ we may assume, 
as before, that $\,\ky=\by$. Then $\,\ky\,$ and 
$\,\gy=-\ky(\,\cdot\,,\hs J\,\cdot\,)$, determine $\,2\byh\nh$, being its real 
and imaginary parts. Combined with the last sentence of the preceding 
paragraph, this completes the proof.\qed
\end{proof}
The next fact concerns two mappings, 
$\,\mathrm{rec}:\ep\smallsetminus\{0\}\to\fg^{\odot2}$ and 
$\,\fg^{\odot2}\nh\ni\my\mapsto\my\hh_\flat\w\in\mathrm{End}\,\fg$. The former 
sends every nonzero element of $\,\ep\,$ (which, as we know, is nondegenerate) 
to its reciprocal. The latter is the operator of index-lowering via $\,\by$, 
and takes values in the space of $\,\bbR\hn$-lin\-e\-ar en\-do\-mor\-phisms of 
$\,\fg$, which include complex multiples of $\,\mathrm{Id}$. We then have
\begin{equation}\label{rec}
\{[\mathrm{rec}\hh(\sy)]\hh_\flat\w:\sy\in\ep\smallsetminus\{0\}\}\,
=\,\{a\hs\mathrm{Id}:a\in\bbC\smallsetminus\{0\}\}\hh.
\end{equation}
Namely, under index raising with the aid of $\,\by$, the operators 
$\,A=a\hs\mathrm{Id}$, for $\,a\in\bbC\smallsetminus\{0\}$, correspond to 
elements $\,\my\,$ of $\,\fg^{\odot2}$ characterized by 
$\,\my^{pq}=\by^{\nh pr}\nh A_r^q$. Every such $\,\my\,$ is in turn the 
reciprocal of $\,\sy\in[\fg\nh^*]^{\odot2}$ defined by 
$\,\sy_{pq}\w=H_{\nh p}^k\by_{kq}\w$, where $\,H=A\nh^{-1}$ (as 
$\,\sy_{pq}\w\my^{sq}=H_{\nh p}^k\by_{ks}\w\by^{\nh sr}\nnh A_r^q
=H_{\nh p}^r\hs A_r^q=\vd_p^q$). Symmetry of $\,\my$, and hence $\,\sy$, is 
obvious from $\,\by$-self-ad\-joint\-ness of $\,A$. The inverses $\,H\,$ of 
our operators $\,A=a\hs\mathrm{Id}\,$ range over nonzero complex multiples of 
$\,\mathrm{Id}$ as well, and so the resulting symmetric forms $\,\sy\,$ act on 
$\,\x,\y\in\fg\,$ by $\,\sy(\x,\y)=\by(u\x+vJ\x,\y)$, where $\,(u,v)\,$ range 
over $\,\bbR\smallsetminus\{0\}$. Therefore $\,\sy=u\by-v\gy$, as required.
\begin{rem}\label{unqdt}{\rm 
The relation $\,\gy=-\by(\,\cdot\,,\hs J\,\cdot\,)\,$ for 
$\,\by=2\,\mathrm{Re}\,\byh$ and $\,\gy=2\,\mathrm{Im}\,\byh$ shows that, once 
$\,J\,$ is fixed, $\,\mathrm{Re}\,\byh$ uniquely determines $\,\byh\nh$. 
Similarly, $\,\mathrm{Re}\,\cf\hh^\fh$ and $\,J\,$ determine the Car\-tan 
\hbox{three\hh-}\hskip0ptform $\,\cf\hh^\fh$ of a complex Lie algebra $\,\fh$, 
cf.\ (\ref{oeb}). In fact, $\,\mathrm{Im}\,\cf\hh^\fh\nh
=-\hh\mathrm{Re}\,\cf\hh^\fh\nh(\,\cdot\,,\,\cdot\,,\hs J\,\cdot\,)$.
}\end{rem}
\begin{rem}\label{cbdet}{\rm 
The bracket $\,[\hskip1.6pt,\hskip.6pt]\,$ of a real\hs$/$com\-plex 
sem\-i\-sim\-ple Lie algebra is uniquely determined by $\,\cf\,$ and $\,\by\,$ 
via (\ref{oeb}). Knowing $\,\cf\,$ and the set of nonzero scalar multiples of 
$\,\by$, rather than $\,\by\,$ itself, makes $\,[\hskip1.6pt,\hskip.6pt]\,$ 
unique up to multiplications by cubic roots of $\,1$. Such factors must 
be allowed as multiplying $\,[\hskip1.6pt,\hskip.6pt]\,$ by a scalar 
$\,r\hs$ replaces $\,\by\,$ and $\,\cf\,$ with $\,r\hh^2\nh\by\,$ and 
$\,r\hh^3\cf$.
}\end{rem}
\begin{rem}\label{cnjug}{\rm 
In the first sentence of Remark~\ref{cbdet}, treating $\,\cf\,$ and $\,\by\,$ 
formally, we see that in the complex case 
$\,\hskip1.3pt\overline{\hskip-1.3pt\cf}\,$ and 
$\,\hskip1.3pt\overline{\hskip-1.3pt\by\hskip-.4pt}\,$ determine, via (\ref{oeb}), the 
same bracket $\,[\hskip1.6pt,\hskip.6pt]\,$ as $\,\cf\,$ and $\,\by$.
}\end{rem}

\section{Proof of Theorem~\ref{autcf}}\label{pa}
\setcounter{equation}{0}\setcounter{thm}{0}
For a real\hs$/$com\-plex Lie algebra $\,\fg$, let the mapping 
$\,\Phi:[\fg\nh^*]^{\wedge3}\nnh\times\hn\fg^{\odot2}\nnh
\to[\fg\nh^*]^{\wedge4}$ be defined by $\,[\Phi(\cf,\my)](\x,\y,\z,\z\hh'\hh)\,
=\,\my(\cf(\x,\y),\cf(\z,\z\hh'\hh))\hs+\hs\my(\cf(\y,\z),\cf(\x,\z\hh'\hh))\hs
+\hs\my(\cf(\z,\x),\cf(\y,\z\hh'\hh))$, where $\,\my\in\fg^{\odot2}$ is 
treated as a symmetric real\hs$/$com\-plex-bi\-lin\-e\-ar form on 
$\,\fg\nh^*\nnh$, and $\,\cf(\x,\y)\,$ stands for the element 
$\,\cf(\x,\y,\,\cdot\,)\,$ of $\,\fg\nh^*\nnh$. If $\,\fg\,$ is also 
sem\-i\-sim\-ple, the isomorphic identification $\,\fg\approx\fg\nh^*$ 
provided by the Kil\-ling form $\,\by\,$ induces an iso\-mor\-phism 
$\,[\fg\nh^*]^{\odot2}\nh\to\fg^{\odot2}\nnh$, which we write as 
$\,\sy\mapsto\sy^\sharp\nh$. Then, in view of (\ref{lgg}) and (\ref{oeb}), 
\begin{equation}\label{fcs}
\Phi(\cf,\sy^\sharp)\,=\,\m\hskip-1.1pt\sy\hskip9pt\mathrm{for\ any}\hskip5pt
\sy\in\fg^{\odot2}\hskip4pt\mathrm{and\ the\ Car\-tan\ 
three}\hyp\mathrm{form}\hskip5pt\cf\hh.
\end{equation}
Theorem~\ref{autcf} is a trivial consequence of the following result combined 
with Lemma~\ref{thrdm}(ii) and the fact that, by multiplying a 
\hbox{Lie\hh}-\hskip0ptal\-ge\-bra bracket operation 
$\,[\hskip1.6pt,\hskip.6pt]\,$ by a nonzero scalar, one obtains a 
\hbox{Lie\hh}-\hskip0ptal\-ge\-bra structure iso\-mor\-phic to the original 
one.
\begin{lem}\label{dtrmn}In a real or com\-plex sem\-i\-sim\-ple Lie algebra\/ 
$\,\fg\hh$, the Car\-tan\/ \hbox{three\hh-}\hskip0ptform and the 
vec\-tor-space structure of\/ $\,\fg\,$ uniquely determine each of the 
following objects.
\begin{enumerate}
  \def\theenumi{{\rm\alph{enumi}}}
\item[{\rm(a)}] The vector subspaces constituting the simple direct summand 
ideals of\/ $\,\fg\hh$.
\item[{\rm(b)}] Up to a sign, in the real case, the complex structure, 
defined as in Section~\/{\rm\ref{sf}}\hh, of every summand ideal\/ $\,\fg'$ 
with\/ $\,\dimr\hn\fg'\nh\ne6\,$ which is a complex Lie algebra, treated as 
real.
\item[{\rm(c)}] Up to multiplications by cubic roots of\/ $\,1$, the 
restrictions of the Lie-\hh al\-ge\-bra bracket of\/ $\,\fg\,$ to all such 
summands of dimensions other than\/ $\,3\,$ or\/ $\,6\hh$.
\item[{\rm(d)}] The Lie-\hh al\-ge\-bra iso\-mor\-phism types of all 
summand ideals\/ $\,\fg'$ with\/ $\,\dimr\hn\fg'\nh\ne3$.
\end{enumerate}
\end{lem}
\begin{proof}Let $\,\cf\,$ be the Car\-tan \hbox{three\hh-}\hskip0ptform of 
$\,\fg$. By (\ref{fcs}), 
$\,\mathrm{Ker}\,\hs\Delta=\{\sy^\sharp:\sy\in\mathrm{Ker}\hskip2.7pt\m\}\,$ 
for the real\hs$/$com\-plex-lin\-e\-ar operator 
$\,\Delta:\fg^{\odot2}\nnh\to[\fg\nh^*]^{\wedge4}$ given by 
$\,\Delta\my=\Phi(\cf,\my)$. Then, if one views all 
$\,\my\in\mathrm{Ker}\,\hs\Delta\subset\fg^{\odot2}$ as linear 
operators $\,\my:\fg\nh^*\nh\to\fg$,
\begin{enumerate}
  \def\theenumi{{\rm\alph{enumi}}}
\item[{\rm(e)}] {\em the simple direct summands of\/ $\,\fg\,$ are 
precisely the minimal elements, in the sense of inclusion, of the set\/} 
$\,\hs\mathbf{S}\hs
=\{\my(\fg\nh^*):\my\in\mathrm{Ker}\,\hs\Delta\hh,\,\mathrm{\ and\ 
}\,\dim\my(\fg\nh^*)=\hs3\,\mathrm{\ or}\,\,\dim\my(\fg\nh^*)\ge\hs6\}$.
\end{enumerate}
In fact, $\,\mathbf{S}\,$ consists of the images of those linear 
en\-do\-mor\-phisms $\,\Sigma:\fg\to\fg\,$ which correspond via (\ref{oms}.b) 
to elements $\,\sy\,$ of $\,\mathrm{Ker}\hskip2.7pt\m\nh$, and have 
$\,\mathrm{rank}\,\Sigma\notin\{0,1,2,4,5\}$. To describe all such $\,\Sigma$, 
we use the four parts of Theorem~\ref{solut}, referring to them as (i) -- 
(iv). Specifically, by (i), our $\,\Sigma$ are direct sums of linear 
en\-do\-mor\-phisms $\,\Sigma_i\w$ of the simple direct summands $\,\fyi\w$ 
of $\,\fg$, while $\,\Sigma_i\w$ are themselves subject to just two 
restrictions: one due to the exclusion of ranks $\,0,1,2,4$ and $\,5$, the 
other depending, in view of (ii) -- (iv), on $\,d_i\w=\dim\fyi\w$, as 
follows. If $\,d_i\w=3$, (ii) states that $\,\Sigma_i\w$ is only required to 
be $\,\by$-self-ad\-joint (to reflect symmetry of $\,\sy_i\w$ related to 
$\,\Sigma_i\w$ as in (\ref{oms}.b)). Similarly, it is clear from (iv) and 
(\ref{und}) that, with a specific the scalar field $\,\bbF\nh$,
\begin{equation}\label{smi}
\begin{array}{l}
\Sigma_i\w\,\mathrm{\ is\ a\ nonzero\ }\,\bbF\hn\hyp\mathrm{multiple\ of\ 
}\,\hs\mathrm{Id}\hs\,\mathrm{\ when\ }\,\hs d_i\w\notin\{3,6\}\hh,
\mathrm{\ where\ }\,\bbF=\bbC\,\mathrm{\ if\ }\,\hs\fyi\hs\mathrm{\ is}\\
\mathrm{either\ complex\ or\ real\ of\ type\ (\ref{rcp}.b),\ and\ 
}\,\bbF=\bbR\,\mathrm{\ for\ real\ }\,\fyi\mathrm{\ of\ type\ (\ref{rcp}.a).}
\end{array}
\end{equation}
In the remaining case, $\,d_i\w=6$. Then, by (iii) and 
Theorem~\ref{solut}(iii), $\,\Sigma_i\w$ is com\-plex-lin\-e\-ar and 
$\,\by$-self-ad\-joint, cf.\ (\ref{cxl}) and (\ref{lrs}.i), but otherwise 
arbitrary.

The image $\,\Sigma(\fg)\,$ of any $\,\Sigma\,$ as above is the direct sum of 
the images of its summands $\,\Sigma_i\w$, and so it can be minimal only if 
there exists just one $\,i\,$ with $\,\Sigma_i\w\ne0$. For this $\,i$, 
minimality of $\,\Sigma(\fg)=\Sigma_i\w\nnh(\fyi\w)\,$ implies that 
$\,\Sigma(\fg)=\fyi\w$. In fact, in view of the last paragraph, the cases 
$\,d_i\w=3$ and $\,d_i\w\notin\{3,6\}\,$ are obvious (the former since 
$\,\mathrm{rank}\,\Sigma_i\w\ge3$) while, if $\,d_i\w=6$, 
com\-plex-lin\-e\-ar\-i\-ty of $\,\Sigma_i\w$ precludes not just $\,0,1,2,4\,$ 
and $\,5$, but also $\,3\,$ from being its real rank.

We thus obtain one of the inclusions claimed in (e): every minimal element of 
$\,\mathbf{S}\,$ equals some summand $\,\fyi\w$. Conversely, any fixed summand 
$\,\fyi\w$ is an element of $\,\mathbf{S}$, realized by $\,\Sigma$ with 
$\,\Sigma_i\w=\mathrm{Id}\,$ and $\,\Sigma\nh_j\w=0\,$ for all $\,j\ne i$, 
cf.\ Lemma~\ref{thrdm}(iii). Minimality of $\,\fyi\w$ is in turn obvious from 
(\ref{smi}) if $\,d_i\w\notin\{3,6\}$, while for $\,d_i\w=3\,$ or 
$\,d_i\w=6\,$ it follows from the restriction on $\,\mathrm{rank}\,\Sigma$ 
combined, in the latter case, with com\-plex-lin\-e\-ar\-i\-ty of 
$\,\Sigma_i$. This yields (e).

Now (a) is obvious from (e), as $\,\Delta\,$ and $\,\hs\mathbf{S}\,$ depend 
only on $\,\cf\,$ and the vec\-tor-space structure of $\,\fg$. To prove 
(b) -- (c), we fix $\,i\,$ with $\,d_i\w\notin\{3,6\}$. Elements $\,\my\,$ of 
$\,\mathrm{Ker}\,\hs\Delta$ having $\,\my(\fg\nh^*)=\fyi\w$ correspond, via 
(\ref{oms}.b) followed by the assignment $\,\sy\mapsto\my=\sy^\sharp\nh$, to 
en\-do\-mor\-phisms $\,\Sigma\,$ of $\,\fg\,$ which satisfy (\ref{smi}) and 
vanish on $\,\fy\nh_j\w$ for $\,j\ne i$. Any such $\,\my$, now viewed as a 
bi\-lin\-e\-ar form on $\,\fg\nh^*\nnh$, is therefore obtained from a 
bi\-lin\-e\-ar form $\,\my_i\w$ on $\,\fyi^{\nh*}$ by the trivial extension to 
$\,\fg\nh^*\nnh$, that is, pull\-back under the obvious 
restriction operator $\,\fg\nh^*\nh\to\fyi^{\nh*}\nnh$.

If $\,\bbF=\bbR$, it is immediate from (\ref{smi}) that the resulting 
forms $\,\my_i\w$ are nonzero multiples of the reciprocal of the Kil\-ling 
form of $\,\fyi\w$, and Remark~\ref{cbdet} implies (c). Next, let 
$\,\bbF=\bbC$. We denote $\,\fyi\w$ treated as a complex Lie algebra by 
$\,\fh\,$ and Car\-tan \hbox{three\hh-}\hskip0ptform of $\,\fh\,$ by 
$\,\cf\hh^\fh\nnh$. Formula (\ref{rec}) states that, in view of (\ref{smi}), 
the reciprocals of our $\,\my_i\w$ are precisely the nonzero elements of the 
space $\,\ep\,$ defined in Section~\ref{sf}. Thus, Lemma~\ref{dtcpx}, 
(\ref{lrs}.ii) and Remark~\ref{unqdt} imply that $\,\cf\,$ determines the 
triple $\,(J,\byh\nnh,\cf\hh^\fh\nh)\,$ uniquely up to replacements by 
$\,(J,a\hn\byh\nnh,a\hh\cf\hh^\fh\nh)$ or 
$\,(-J,a\hskip.9pt\overline{\hskip-1.3pt\byh}\nh,
a\hskip1.7pt\overline{\hskip-1.3pt\cf\hh^\fh})$, with 
$\,a\in\bbC\smallsetminus\{0\}$. This proves (b), while using 
Remarks~\ref{cbdet} --~\ref{cnjug} we obtain (c) for $\,\bbF=\bbC\,$ as well. 
Finally, (c) and Lemma~\ref{thrdm}(i)\hh--\hh(iii) easily yield (d).\qed
\end{proof}

\begin{acknowledgements}
We thank Robert Bryant and Nigel Hitchin for helpful comments about 
Theorem~\ref{autcf}.
\end{acknowledgements}

\section*{Appendix: \ Meyberg's theorem}
For any complex simple Lie algebra $\,\fg$, the operator $\,\cu\,$ with 
(\ref{oms}) is di\-ag\-o\-nal\-izable. Its systems $\,\mathrm{Spec}\hs[\fg]\,$ 
of eigen\-val\-ues and $\,\mathrm{Mult}\hs[\fg]\,$ of the corresponding 
multiplicities are
\[
\begin{array}{l}
\mathrm{Spec}\hs[\fsl_n\w]=(2,1,2/n,-2/n)\hskip5pt\mathrm{and}\\
\mathrm{Mult}\hs[\fsl_n\w]=(1,n^2-1,n^2(n-3)(n+1)/4,n^2(n+3)(n-1)/4)\hh,
\hskip5pt\mathrm{if\ }\,n\ge4\hh.\\
\mathrm{Spec}\hs[\fsp_n\w]=(2,(n+4)/(n+2),-4/(n+2),2/(n+2))\hskip6pt
\mathrm{for\ even\ }\,n\ge4\hh,\hskip5pt\mathrm{and}\\
\mathrm{Mult}\hs[\fsp_n\w]
=(1,(n-2)(n+1)/2,n(n+1)(n+2)(n+3)/24,n(n-1)(n-2)(n+3)/12)\hh.\\
\mathrm{Spec}\hs[\fso_n\w]=(2,(n-4)/(n-2),4/(n-2),-2/(n-2))\hskip5pt
\mathrm{if\ }\,n=7\,\mathrm{\ or\ }\,\,n\ge9\hh,\hskip5pt\mathrm{while}\\
\mathrm{Mult}\hs[\fso_n\w]
=(1,(n+2)(n-1)/2,n(n-1)(n-2)(n-3)/24,n(n+1)(n+2)(n-3)/12)\\
\end{array}
\]
and, if $\,\fg\,$ is one of the exceptional complex Lie algebras 
$\,\fsl_2\w,\hh\fsl_3\w,\hh\fg_2\w,\hh\fso_8\w,\hh\mathfrak{f}_4\w,
\hh\mathfrak{e}_6\w,\hh\mathfrak{e}_6\w,\hh\mathfrak{e}_7\w,
\hh\mathfrak{e}_8\w$,
\begin{equation}\label{exc}
\begin{array}{l}
\mathrm{Spec}\hs[\mathfrak{g}]=(2,(1+w)/6,(1-w)/6),\hskip5pt\mathrm{with}
\hskip5pt\mathrm{Mult}\hs[\mathfrak{g}]\hskip5pt\mathrm{equal\ to}\\
(1,3d[(d+2)w-(d+32)]/[w(11-w)],3d[(d+2)w+(d+32)]/[w(11+w)])\hh,
\end{array}
\end{equation}
where $\,d=\dim\fg\,$ and $\,w=[(d+242)/(d+2)]^{1/2}\nnh$. This is a result of 
Meyberg \cite{meyberg} who, instead of our $\,\cu$, studied the operator 
$\,T=\cu/2$. (The exponent $\,1/2\,$ is missing in \cite{meyberg}). For 
$\,\fsl_2\w$, the resulting ``eigen\-value'' $\,4/3\,$ of multiplicity $\,0\,$ 
should be disregarded. All iso\-mor\-phism types of complex simple Lie 
algebras are listed above, cf.\ Remark~\ref{clssf}. 

The dimensions $\,d\,$ of 
$\,\fsl_2\w,\hh\fsl_3\w,\hh\fg_2\w,\hh\fso_8\w,\hh\mathfrak{f}_4\w,
\hh\mathfrak{e}_6\w,\hh\mathfrak{e}_6\w,\hh\mathfrak{e}_7\w,
\hh\mathfrak{e}_8\w$ are $3,\hs8,14,\hs28,\hs52,\hs78,\hs133,\hs248\,$ 
\cite[p.\ 21]{iachello}. The eigenvalues $\,0\,$ and $\,1\,$ in (\ref{exc}) 
would correspond to $\,w=1\,$ or $\,w=5$, of which only the latter occurs, for 
$\,d=8\,$ and $\,\fg=\fsl_3\w$, in agreement with Lemma~\ref{eigen}.

\end{document}